\documentclass{article}
\usepackage{amsmath,amsfonts,amsthm,amssymb,mathrsfs,latexsym,paralist, xypic}

\usepackage{tikz}
\usetikzlibrary{arrows,automata}

\tikzstyle{V}=[fill=black,circle,scale=0.4, outer sep = 4pt]

\newtheorem{thm}{Theorem}[section]
\newtheorem{prop}[thm]{Proposition}

\theoremstyle{remark}
\newtheorem{rmk}[thm]{Remark}
\newtheorem{example}[thm]{Example}
\newtheorem{conj}[thm]{Conjecture}
\theoremstyle{definition}
\newtheorem{defn}[thm]{Definition}

\DeclareMathOperator{\Ad}{Ad}
\DeclareMathOperator{\Aut}{Aut}

\newcommand{\z}{^{(0)}}
\renewcommand{\2}{^{(2)}}
\newcommand{\inv}{^{-1}}

\newcommand{\bi}{\begin{itemize}}
\newcommand{\ei}{\end{itemize}}
\newcommand{\be}{\begin{enumerate}}
\newcommand{\ee}{\end{enumerate}}

\newcommand{\Q}{\mathbb{Q}}
\newcommand{\C}{\mathbb{C}}

\newcommand{\T}{\mathbb{T}}
\renewcommand{\H}{\mathcal{H}}

\newcommand{\G}{\mathcal{G}}
\newcommand{\K}{\mathcal{K}}

\newcommand{\W}{\mathcal{R}}

\newcommand{\R}{\mathbb{R}}
\newcommand{\N}{\mathbb{N}}
\newcommand{\Z}{\mathbb{Z}}

\begin{document}

\title{Twists over \'etale groupoids and twisted vector bundles}
\author{Carla Farsi and Elizabeth Gillaspy}
\date{\today}
\maketitle


\begin{abstract}
Inspired by recent papers on twisted $K$-theory, we consider in this article the question of when a twist $\mathcal{R}$ over a locally compact Hausdorff groupoid $\mathcal{G}$ (with unit space a CW-complex) admits a twisted vector bundle, and we relate this question to the Brauer group of $\mathcal{G}$.  
 We show that the twists which admit twisted vector bundles give rise to a subgroup of the Brauer group of $\mathcal{G}$.  When $\mathcal{G}$ is an \'etale groupoid, we establish conditions (involving the classifying space $B\mathcal{G}$ of $\mathcal{G}$) which imply that a torsion twist $\mathcal{R}$ over $\mathcal{G}$ admits  a twisted vector bundle.
\end{abstract}

\section{Introduction}

$C^*$-algebras associated to dynamical systems have provided motivation and examples for a wide array of topics in $C^*$-algebra theory: representation theory, ideal structure, 
$K$-theory, 
 classification, 
 and connections with mathematical physics, to name a few.
  In many of these cases, a complete understanding of the theory has required expanding the notion of a dynamical system to allow for partial actions 
  and twisted actions, 
   as well as actions of group-like objects such as semigroups or groupoids.

For example, the $C^*$-algebras $C^*(\G; \W)$ associated to a groupoid $\G$ and a twist $\W$ over $\G$ (hereafter referred to as \emph{twisted groupoid $C^*$-algebras}) provide important insights into mathematical physics as well as the structure of other $C^*$-algebras.  First, the collection of twists $\W$ over a groupoid $\G$ is intimately related with the cohomology of $\G$, cf.~\cite{equiv-sheaf-coh, brauer-gp-gpoids, tu-cohlogy}.
Another structural result is due to Kumjian \cite{c*-diagonals} and Renault \cite{renault-cartan}:  groupoid twists classify Cartan pairs. 
Finally, the papers \cite{TXLG, bouwk-mathai,
BCMMS} establish that twisted groupoid $C^*$-algebras  classify $D$-brane charges in many flavors of string theory.

We also note, following \cite{equiv-disint, TXLG}, that groupoid twists constitute an example of Fell bundles. Indeed, Fell bundles provide a universal framework for studying all of the generalized dynamical systems mentioned above.


In several recent papers (cf.~\cite{bema, TXLG, emerson-meyer}) on twisted groupoid $C^*$-algebras, the $K$-theory groups of these $C^*$-algebras have received a good deal of attention.
  Of particular interest is the question of when $K_0(C^*(\G; \W))$  can be completely understood in terms of $\G$-equivariant vector bundles.  Phillips established in Chapter 9 of \cite{phillips-equivar-bk} that $\G$-equivariant vector bundles may not suffice to describe $K_0(C^*(\G; \W))$, even when $\G  = M \rtimes G$ is a transformation group and $\W$ is trivial.  Vector bundles provide a highly desirable geometric perspective on $K_0(C^*(\G; \W))$, however, and so conditions are sought (cf.~\cite{adem-ruan, BCMMS, cantarero-equiv, dwyer, emerson-meyer, luck-oliver}) under which $K_0(C^*(\G; \W))$ is generated by $\G$-equivariant vector bundles.

%
%

  In Theorem 5.28 of \cite{TXLG},  Tu, Xu, and Laurent-Gengoux study this question for proper Lie groupoids $\G$.  They establish, in this context,  sufficient conditions for the $K$-theory group $K_0(C^*(\G, \W))$ associated to a twist $\W$ over $\G$ to be generated by $(\W, \G)$-twisted vector bundles over the unit space of $\G$ (see Definition \ref{def-twisted-vector-bundles} below).  
  A necessary condition is that $\W$ be a torsion element of the Brauer group of $\G$. Conjecture 5.7 on page 888 of \cite{TXLG} states that, if $\G$ is a proper Lie groupoid acting cocompactly on its unit space, then this condition is also sufficient. 

Conjecture 5.7 of \cite{TXLG} has not yet been disproved, but it has only been proven true  in certain special cases: cf.~\cite{luck-oliver, emerson-meyer, cantarero-equiv} 
when $\W = \G \times \T$ is the trivial twist,  
\cite{BCMMS} for nontrivial twists $\W$ over manifolds $M$, and \cite{adem-ruan, dwyer, lupercio-uribe} for 
 nontrivial twists over representable orbifolds $G \rtimes M$, where $G$ is a discrete group acting properly on a compact space $M$.

In hopes of shedding more light on this Conjecture, we present an equivalent formulation in Conjecture \ref{conj} below, using the Brauer group of $\G$ as defined in \cite{brauer-gp-gpoids}. 
Our reformulated conjecture relies on our result (Proposition \ref{prop:tw-tau-Brauer}) that, for any locally compact Hausdorff groupoid $\G$ whose unit space is a CW-complex, the collection of twists $\W$ over $\G$ which admit twisted vector bundles gives rise to a subgroup $Tw_\tau(\G)$ of the Brauer group $\text{Br}(\G)$.  

We note that  Theorem 3.2 of \cite{karoubi-twisted} also establishes a link between twisted vector bundles and the Brauer group, but Karoubi's approach in \cite{karoubi-twisted} differs substantially from ours, and does not address the group structure of $Tw_\tau(\G)$.

In the second part of the paper, we address the question of when a torsion twist $\W$ over an \'etale groupoid $\G$ admits a twisted vector bundle. The existence of such vector bundles is necessary (but not sufficient) in order for $K_0(C^*(\G;\W))$ to be generated by twisted vector bundles.  

Theorem \ref{main-thm} below establishes
  that if the classifying space $B\G$ is a compact CW-complex and if a certain principal $PU(n)$-bundle $P$ lifts to a $U(n)$-principal bundle $\tilde{P}$, then  up to Morita equivalence, the torsion twist $\W$ admits a twisted vector bundle.
To our knowledge, the connection between classifying spaces and twisted vector bundles has not been explored previously in the literature; we are optimistic that Theorem \ref{main-thm} will lead to new insights into the $K$-theory of twisted groupoid $C^*$-algebras.


\subsection{Structure of the paper}
We begin in Section \ref{sec:backgrounddefs} by reviewing the basic concepts we will rely on throughout this paper: locally compact Hausdorff groupoids, twists over such groupoids, groupoid vector bundles and twisted vector bundles.  In Section \ref{sec-Twisted-and-Brauer} we show that, for any locally compact Hausdorff groupoid $\G$ whose unit space is a CW-complex, the collection of twists over $\G$ which admit twisted vector bundles gives rise to a subgroup of the Brauer group of $\G$, and we use this to present an alternate formulation of Conjecture 5.7 from \cite{TXLG}. Finally in Section \ref{sec-twist-etale} we consider torsion twists for \'etale groupoids.  We establish, in Theorem \ref{main-thm}, sufficient conditions for a torsion twist $\W$ over an \'etale groupoid $\G$ to admit a twisted vector bundle, and we present examples showing that the hypotheses of Theorem \ref{main-thm} are satisfied in many cases of interest.

\subsection{Acknowledgments:} The authors are indebted to Alex Kumjian for pointing out a flaw in an earlier version of this paper. We would also like to thank 
Angelo Vistoli for helpful correspondence.

\section{Definitions}
\label{sec:backgrounddefs}
Recall that a \emph{groupoid} is a small category with inverses.  Throughout this note, $\G$ will denote (the space of arrows of) a groupoid with unit space $\G^{(0)} $, with  source, range (or target), and unit maps 
\[
s, r :\G \longrightarrow \G\z,\ 
u : \G\z \longrightarrow \G. 
\]
As usual we denote the set of  composable elements of $\G$   by $\G^{(2)}$, where
\[
\G^{(2)} =  \G \times_{s,\G\z,t} \G= \{ (g_1, g_2 )\in \G \times  \G \ | \ s(g_1)=r(g_2)  \}.
\]

In this paper we will primarily be concerned with \emph{locally compact Hausdorff groupoids}.  These are groupoids $\G$ such that the spaces $ \G\z, \G, \G\2$ have locally compact Hausdorff topologies with respect to which the maps $s, r: \G \to \G\z$, the multiplication $\G\2 \to \G$, and the inverse map $\G \to \G$ are continuous. 
Conjecture \ref{conj} below makes reference to \emph{Lie groupoids}, which are locally compact Hausdorff groupoids such that the spaces $ \G\z, \G, \G\2$ are smooth manifolds and all of the structure maps between them are smooth.

Theorem \ref{main-thm} deals with \emph{\'etale groupoids}, which are locally compact Hausdorff groupoids $\G$ for which $r,s$ are local homeomorphisms. For example, if a discrete group $\Gamma$ acts on a CW-complex $M$, the associated transformation group $\Gamma \ltimes M$ is an \'etale groupoid.

\begin{defn} 
\label{def-morph-of-lie-groupoids}
Let $\G_1 ,\G_2 $  be two  locally compact Hausdorff  groupoids with unit spaces $\G\z_1, \G\z_2$ respectively. A morphism 
$
f : \G_1  \rightarrow \G_2 $
consists of a pair of continuous maps $
f=(f_0, f_1 ),$ with
\[
f_0: \G\z_1 \to \G\z_2,\ f_1: \G_1 \to \G_2,
\]
such that, if we denote by $s_{\G_j}$ and  $r_{\G_j}$ the source and range maps of $\G_j$, $j=1,2$, we have
\[
s_{\G_2} \circ f_1\  =\ f_0 \circ s_{\G_1}, \ and \  r_{\G_2} \circ f_1\  =\ f_0 \circ r_{\G_1}.
\]
\end{defn}

The notion of a twist or $\T$-central extension of a groupoid $\G$ was originally developed
(cf.~\cite{c*-diagonals, cts-trace-gpoid-II, TXLG})  to provide a ``second cohomology group'' for groupoids.  Groupoid twists and their associated twisted vector bundles (see Definition \ref{def-twisted-vector-bundles} below) are the groupoid analogues of 
group 2-cocycles and projective representations.

\begin{defn} 
\label{def-twist-of-lie-groupoids}

Let $\G$ be a locally compact Hausdorff groupoid with unit space $\G\z$. A \emph{$\T$-central extension (or ``twist'')} of
$\G $  consists of
\begin{enumerate}
\item  A locally compact Hausdorff groupoid $\W $ with unit space $\G\z$, together with a morphism of locally compact Hausdorff  groupoids 
\[
(id, \pi) : \W\rightarrow \G
\]
which restricts to the identity on $\G\z$.
\item A left $\T$--action on $\W$, with respect to which $ \W $ is a left  principal  $\T$-bundle over $\G$. 
\item 
These two
structures are compatible in the sense that 
\[
(z_1  r_1)(z_2 r_2) = z_1z_2  (r_1 r_2), \forall\ z_1, z_2 \in \T,\ 
\forall (r_1, r_2) \in  \W^{(2)} = \W \times_{s,\G\z,r} \W.
\]
\end{enumerate}
We write $Tw(\G)$ for the set of twists over $\G$.
\end{defn}

These conditions (1)-(3) imply the exactness of the  sequence of groupoids
\[ \G\z \to \G\z \times \T \to \W \to \G \rightrightarrows  \G\z ,\]
which highlights the parallel between twists over a groupoid $\G$ and extensions of $\G$ by $\T$ (or elements of the second cohomology group $H^2(\G, \T)$).  

If $\W_1, \W_2 \in Tw(\G)$, we can form their Baer sum 
\[\W_1 + \W_2 := \{(r_1, r_2) \in \W_1 \times \W_2: \pi_1(r_1) = \pi_2(r_2)\}/\sim,\]
where $(r_1, r_2) \sim (z r_1, \overline{z} r_2)$ for all $z \in \T$.  
Define an action of $\T$ on $\W_1 + \W_2$ by $z \cdot [(r_1, r_2)] = [(z r_1, r_2)] = [(r_1, zr_2)]$, and observe that with this action, $\W_1 + \W_2$ becomes a twist over $\G$.

With this operation $Tw(\G)$ becomes a 
group; the identity element is the trivial extension $ \G \times \T$, and the inverse of a twist $\W$ is the twist $\overline{\W}$.  As groupoids, $\W = \overline{\W}$; however, the action of $\T$ on $\overline{\W}$ is the conjugate of the action on $\W$.  To be precise, if $r \in \W$, denote by $\overline{r}$ the corresponding element of $\overline{\W}$.  Then 
\[z \cdot \overline{r} = \overline{ \overline{z} \cdot r}.\]

In this note, we will consider actions of groupoids $\G$ and twists $\W$ over $\G$ on a variety of spaces.  We make this concept precise as follows.
\begin{defn} 
\label{def-G-spaces}
Let $\G $ be a locally compact Hausdorff groupoid with unit space $\G\z$. A \emph{$\G$-space} is a locally trivial fiber bundle $J: P \to \G\z
$
such that, 
setting 
\[\G * P = \{(g, p) \in \G \times P: s(g) = J(p)\}\] and equipping $\G * P$ with the subspace topology inherited from $\G \times P$, we have a continuous map 
$\sigma: \G * P \to  P$ satisfying
\begin{itemize}
\item $\sigma(J(p), p) = p$ for all $p \in P$;
\item $J(\sigma(g,p)) = r(g)$ for all $(g, p) \in \G * P$;
\item If $(g,h) \in \G\2$ and $(h, p) \in \G * P$, then $\sigma(g, \sigma(h,p)) = \sigma(gh, p)$.
\end{itemize}
We will often write $g \cdot p$ for $\sigma(g,p) \in P$.

Note that, as a consequence of the above definition, the map $\sigma_g: P_{s(g)} \to P_{r(g)}$ given by $p \mapsto \sigma(g, p)$ must be a homeomorphism, for all $g \in \G$.
\end{defn}

\begin{defn} 
\label{def-twisted-vector-bundles}
\begin{enumerate}
\item Let $\G $ be a locally compact Hausdorff groupoid with unit space $\G\z$, where $\G\z$ is a CW-complex. A \emph{$\G$--vector bundle} is a vector bundle $
J : E \to \G\z
$
which is a $\G$-space in the sense of Definition \ref{def-G-spaces}.


\item Let 
\[
\G\z \to \T \times \G\z \stackrel{i}{\longrightarrow}\W \stackrel{j}{\longrightarrow} \G \rightrightarrows  \G\z
\] 
be a $\T$-central extension of  locally compact Hausdorff groupoids. By
a\emph{ $(\G, \W)$--twisted vector bundle}, we mean a $\W$-vector bundle $J : E \to \G\z  $ such that, whenever $z\in  \T,\ r \in \W,e \in  E$  such that $ s(r) = J(e)$, we have 
\begin{equation}
\label{eq-compat}
(z\cdot r)  \cdot e = z (r \cdot e) .
\end{equation}
Here, the action on the right-hand side of the equation is simply scalar multiplication (identifying $\T$ with the unit circle of $\C$).
\item An equivalent characterization of $(\W, \G)$-twisted vector bundles is the following:

A
$\W$-vector bundle $E \to \G\z$ is a $(\W,\G)$-twisted vector bundle if and only if the subgroupoid $\ker j \cong \G\z \times \T$ of $\W$
acts on $E$ by scalar multiplication, where $\T$ is identified with the unit circle of $\C$.
\end{enumerate}
\end{defn}

In Proposition \ref{prop:tw-tau-Brauer}, we will establish a connection between the twists over $\G$ which admit twisted vector bundles and the Brauer group of $\G$, as introduced in \cite{brauer-gp-gpoids}.  Thus, we review here a few facts about the Brauer group and its connection to $Tw(\G)$.

\begin{defn}
\label{def:brauer-gp}
Let $\G $ be a locally compact Hausdorff groupoid.
As in Definition 8.1 of \cite{brauer-gp-gpoids}, we will denote by $Br_{0}(\G) $  the group of Morita equivalence classes of $\G$-spaces $\mathcal{A}$ such that $\mathcal{A} =  \G\z \times \K(\H)$ for some Hilbert space $\H$.  
We denote the class in $Br_0(\G)$ of $\mathcal{A}$ by $[\mathcal{A}, \alpha]$, where $\alpha$ is the action of $\G$ on $\mathcal{A}$.

Also, let  $\mathcal{E}(\G)$ be the quotient of $Tw(\G)$ by Morita equivalence or, equivalently,  the quotient by the subgroup  $W$ of elements which are Morita equivalent to the trivial twist. See Definition 3.1 and  Corollary 7.3 of \cite{brauer-gp-gpoids} for details. 
\end{defn}
Theorem 8.3 of \cite{brauer-gp-gpoids} establishes that
\[
Br_0(\G) \cong \mathcal{E}(\G)= Tw(\G)/W.
\]

\section{Twisted vector bundles and the Brauer group}
\label{sec-Twisted-and-Brauer}
Let $\G$ be a locally compact Hausdorff groupoid with unit space  a CW-complex. 
In this section, we will show that the subset $Tw_\tau(\G)$ of twists over $\G$ which admit twisted vector bundles gives a subgroup of $Br_0(\G)$.

\begin{defn}
For a locally compact Hausdorff groupoid $\G$, let $Br_{\tau}(\G)$ be the subgroup of $Br_0(\G)$ consisting of Morita equivalence classes $[\mathcal{A}, \alpha]$ of elementary $\G$-bundles  $\mathcal{A}  = \G\z \times \mathcal{ K(H)}$ with zero Dixmier-Douady invariant, such that $\H$ is finite dimensional.

When, in addition, the unit space of $\G$ is a CW-complex, we denote by $Tw_{\tau}(\G)$  the subset of $Tw(\G)$ consisting of twists $\W$ over $\G$ that admit a twisted vector bundle.
\end{defn}

\begin{prop}
Let $\G$ be a locally compact groupoid whose unit space is a CW-complex. 
Then $Tw_{\tau}(\G)$ is a subgroup of $Tw(\G)$. 
\end{prop}

\begin{proof} 
\begin{enumerate}
\item  (Closure under operation) 
Given two twists
\[
\G\z \to \G\z \times \T\stackrel{i_1}{\longrightarrow}\W_1
\stackrel{j_1}{\longrightarrow} \G \rightrightarrows  \G\z
, \qquad 
\G\z \to \G\z \times \T\stackrel{i_2}{\longrightarrow}\W_2
\stackrel{j_2}{\longrightarrow} \G \rightrightarrows  \G\z
\] 
that admit twisted vector bundles  $E_1$ and $E_2$  respectively, it is straightforward to show that
\[
E_1 * E_2 : = \{ (e_1, e_2)\in \ E_1 \oplus  E_2 |  \ J_1(e_1) = J_2(e_2)  \}/\sim,
\] 
is a twisted vector bundle for the Baer sum $\W_1 + \W_2$. 
The action of $\W_1 + \W_2$ on $E_1 * E_2$ is given by 
\[[(r_1, r_2)] \cdot [(e_1, e_2)] = [(r_1 \cdot  e_1, r_2 \cdot  e_2)].\]

\item {(Neutral Element)} The neutral element of $Tw(\G)$ is $\G \times \T $. 
Note that $\G \times \T$ admits a  twisted vector bundle $E$ -- namely, $E=\G\z \times \C$, with the action $(g, z) \cdot (s(g), v) = (r(g), zv)$.

\item {(Inverses)} 
We must show that, if $\W\in Tw_\tau(\G)$, then $\overline{\W} \in Tw_\tau(\G)$.

For $\W \in Tw_\tau(\G)$, let $E \to \G\z$ be a ($\W, \G)$-twisted vector bundle. Write $\overline{E}$ for the conjugate vector bundle -- that is, $\overline{E} = E$ as sets, and the additive operation on $\overline{E}$ agrees with that on $E$ (in symbols, $\overline{e + f} = \overline{e} + \overline{f}$), but the $\C$ action on $\overline{E}$ is the conjugate of the action on $E$: $z \cdot \overline{e} = \overline{ \overline{z} \cdot e}$. 
Define an action of $\overline{\W}$ on $\overline{E}$ by $\overline{r} \cdot \overline{e} = \overline{r \cdot e}$.  This action makes $\overline{E}$ into a $\overline{\W}$-vector bundle since $E$ is a $\W$-vector bundle.  
 Moreover, for any $z \in \T$ we have 
\begin{align*}
(z \cdot \overline{r}) \cdot \overline{e} &= \overline{\overline{z} \cdot r} \cdot \overline{e} = \overline{(\overline{z} r) \cdot e} \\
&= \overline{\overline{z}( r\cdot e)} = z  \overline{r \cdot e} \\
&= z (\overline{r} \cdot \overline{e}).
\end{align*}
Thus, $\T$ acts by scalars on $\overline{E}$, and so $\overline{E}$ is a $(\overline{W}, \G)$-twisted vector bundle.
\end{enumerate}

\end{proof}

\begin{rmk}
Recall from Proposition 5.5 of \cite{TXLG} that if a twist $\W$ over $\G$ admits a twisted vector bundle, then $\W$ must be torsion.  Thus, $Tw_\tau(\G)$ is a subgroup of $Tw^{tor}(\G)$, the torsion subgroup of $Tw(\G)$.
\end{rmk}

\subsection{The image of  $Tw_\tau(\G)$ in $Br_0(\G)$}
Recall that if $\G$ is a locally compact Hausdorff groupoid with unit space $\G\z$, then $Br_0(\G)$ consists of Morita equivalence classes of $\G$-spaces of the form $\mathcal{A} = \G\z \times \K(\H)$.

In Section 8 of \cite{brauer-gp-gpoids}, the authors construct an isomorphism $\Theta: Br_0(\G) \to Tw(\G)/W$, where $W$ is  the subgroup of $Tw(\G)$ consisting of elements which are Morita equivalent to the trivial twist. We will use this isomorphism to study the subgroup of $Br_0(\G)$ corresponding to $Tw_\tau(\G)$. 

Proposition 8.7 of \cite{brauer-gp-gpoids}  describes a homomorphism $\theta: Tw(\G) \to Br_0(\G)$ which induces the inverse of $\Theta$. 


\begin{prop}
\label{prop:tw-tau-Brauer}
Suppose $\G$ is a locally compact Hausdorff groupoid whose unit space $\G\z$ is a connected CW-complex, and suppose $\W \in Tw_\tau(\G)$.  Then there exists a finite-dimensional $\G$-vector bundle $V\to \G\z$ such that $\theta(\W)= [\Aut(V), \alpha]$, where $\alpha$ is induced by the action of $\G$ on $V$.

  Moreover, if $[\mathcal{A}, \alpha'] \in Br_0(\G)$ and $(\mathcal{A}, \alpha')$ is Morita equivalent to $(\mathcal{M}_n, \alpha)$ where $\mathcal{M}_n$ is an $M_n(\C)$-bundle over $\G\z$, 
then $[\mathcal{A}, \alpha'] = [\mathcal{M}_n, \alpha]$ lies in $ \theta(Tw_\tau(\G)).$

In other words, $Br_\tau(\G) \cong Tw_\tau(\G)$.
\end{prop}

\begin{proof}
If $\W \in Tw_\tau(\G)$ and $V$ is a ($\W, \G)$-twisted vector bundle, write $j: \W \to \G$ for the projection of $\W$ onto $\G$ and write $\sigma: \W * V \to V$ for the action of $\W$ on $V$.  Define $\alpha: \G * \text{Aut}(V) \to \text{Aut}(V)$ by 
\[(\alpha(g, A)(v) = \sigma(\eta,  A (\sigma(\eta^{-1}, v))),\]
where $v \in V_{r(g)}$ and $\eta \in j\inv(g)$.

Note that $\alpha(g, A)$ does not depend on our choice of $\eta\in j^{-1}(g)$: If $\eta, \eta' \in j^{-1}(g)$, the fact that $\W$ is a principal $\T$-bundle over $\G$ implies that $\eta = z \eta'$ for some $z \in \T$.  Since $V$ is a $(\G, \W)$-twisted vector bundle, $\sigma(\eta, v) = z \sigma(\eta', v)$, and consequently 
\[ \sigma(\eta,  A(\sigma(\eta^{-1}, v))) = \sigma(\eta', A(\sigma((\eta')^{-1}, v))).\]

Now, Lemma 8.8 of \cite{brauer-gp-gpoids} establishes that $[\text{Aut}(V), \alpha] = \theta(\W)$. 

For the second statement, suppose $\alpha$ is an action of $\G$ on a bundle $\mathcal{M}_n$ of $n$-dimensional matrix algebras over $\G\z$.   Then Theorem 8.3 of \cite{brauer-gp-gpoids} explains how to construct the twist $\Theta([\mathcal{M}_n, \alpha])$, using a pullback construction.  To be precise,
\[ \Theta([\mathcal{M}_n, \alpha]) = \{(g, U) \in \G \times U_n(\C): \alpha_g = \Ad U\} =: \W(\alpha).\]
We will construct a $(\G, \W(\alpha))$-twisted vector bundle, proving that $[\mathcal{M}_n, \alpha] \in \theta(Tw_\tau(\G))$.

The $\T$-action on $\W(\alpha)$ which makes it into a twist over $\G$ is given by \[z \cdot (g, U) = (g, z \cdot U).\]

Consider the sub-bundle $\mathcal{GL}_n$ of $\mathcal{M}_n$ obtained by considering only the invertible elements of $M_n(\C)$ in each fiber of $\mathcal{M}_n$.  Notice that $GL_n(\C)$ acts on $\mathcal{GL}_n$ by right multiplication in each fiber, and that this action is continuous, and free and transitive in each fiber, and hence makes $\mathcal{GL}_n$ into a principal ${GL}_n$ bundle.
 We consequently obtain an associated vector bundle over $\G\z$, 
\[ V = \mathcal{GL}_n \times_{GL_n(\C)} \C^n.\]
Moreover,  $\W(\alpha)$  acts on $V$:
\[(g, U) \cdot [A, v] = [\alpha_g(A), Uv].\]
To see that this action is well defined, take $G \in GL_n(\C)$ and calculate:
\begin{align*}
[\alpha_g(AG), U(G\inv v)] &= [U AGU\inv, UG\inv v] = [UA, v]= [UAU\inv, Uv]\\
[\alpha_g(A), v] &= [UAU\inv, Uv].
\end{align*}
Moreover, 
\begin{align*}
(z \cdot (g, U)) \cdot [A,v] &= [\alpha_g(A), z U(v)]  \\
& = z \cdot [\alpha_g(A), Uv] = z \cdot \left( (g, U) \cdot [A,v]\right),
\end{align*}
so $V$ is an $(\W(\alpha), \G)$-twisted vector bundle.  Thus, $\W(\alpha) \in Tw_\tau(\G)$ whenever $[\alpha, \mathcal{M}] \in Br_0(\G)$.
\end{proof}

Proposition \ref{prop:tw-tau-Brauer} thus establishes that twists $\W$ over $\G$ which admit twisted vector bundles correspond  to $C^*$-bundles over $\G\z$ with finite-dimensional fibers.  Phrased in this way, the parallel between Proposition \ref{prop:tw-tau-Brauer} and Theorem 3.2 of \cite{karoubi-twisted} becomes evident.  However, the two proofs take very different approaches.  Moreover, Karoubi does not address the group structure of $Tw_\tau(\G)$ in Theorem 3.2 of \cite{karoubi-twisted}.

Proposition \ref{prop:tw-tau-Brauer} also
allows us to rephrase  Conjecture 5.7 of \cite{TXLG} in terms of the Brauer group, as follows.  Recall that, in its original form, Conjecture 5.7 of \cite{TXLG} asserts that all torsion elements of $Tw(\G)$ should admit twisted vector bundles, if $\G$ is proper and the quotient $\G\z/\G$ is compact.

\begin{conj}[\cite{TXLG} Conjecture 5.7]
\label{conj}
Let $\G$ be a proper Lie groupoid such that the quotient $\G\z /\G$ is compact, and 
let $[\mathcal{A}, \alpha] \in \text{Br}_0(\G)$ be a torsion element.  Then $[\mathcal{A}, \alpha] =[\mathcal{M}, \alpha']$ for some finite-dimensional matrix algebra bundle $\mathcal{M}$ over $\G\z$ and an action $\alpha'$ of $\G$ on $\mathcal{M}$.
\end{conj}

\section{Twisted vector bundles for \'etale groupoids}
\label{sec-twist-etale}

 In this section we consider torsion twists over \'etale groupoids $\G$. 
We establish in Theorem \ref{main-thm} sufficient conditions for a torsion twist $\W$ over $\G$ to admit (up to Morita equivalence) a twisted vector bundle, and we describe examples meeting these conditions in Section \ref{sec:examples}.  The conditions of Theorem \ref{main-thm} are phrased in terms of the classifying space $B\G$ and in terms of a principal bundle $P$  associated 
 to $\W$.  Using $B\G$ to study twisted vector bundles appears to be a new approach; this perspective was inspired by Moerdijk's result in \cite{moerdijk-haefliger} identifying $H^*(\G, \mathscr{S})$ and $H^*(B\G, \tilde{\mathscr{S}})$ for an abelian $\G$-sheaf $\mathscr{S}$, and the Serre-Grothendieck Theorem (cf. Theorem 1.6 of \cite{grothendieck}) relating $H^1(M, PU(n))$ and $H^2(M, \T)$ for $M$ a CW-complex.

 
%
%
%
%
 We begin with some preliminary definitions and results.


%
%
 \begin{defn}
 Let $\G$ be a topological groupoid.  The \emph{simplicial space associated to $\G$} is 
 \[\G_\bullet  = \{\G^{(k)}, \epsilon^k_j, \eta^j_k\}_{0 \leq j\leq k\in \N},\]
  where $ \G^{(k)}$ is the space of composable $n$-tuples in $\G$, $\epsilon^k_j: \G^{(k)} \to \G^{(k-1)}$, and $\eta^j_k: \G^{(k)} \to \G^{(k+1)}$ are given as follows:
\begin{align*}
\epsilon^k_0(g_1, \ldots, g_k) &= (g_2, \ldots, g_k) \\
\epsilon^k_i(g_1, \ldots, g_k) &= (g_1, \ldots, g_{i} g_{i+1}, \ldots, g_k) \text{ if } 1 \leq i \leq k-1 \\
\epsilon^k_k (g_1, \ldots, g_k) &= (g_1, \ldots, g_{k-1}) 
\end{align*}
If $k=1$, we have $\epsilon^1_0(g) = s(g),\ \epsilon^1_1(g) = r(g)$. 
 
 The degeneracy maps $\eta^k_i$ are given for $k \geq 1$ by 
 \begin{align*}
 \eta^k_i(g_1, \ldots, g_k) &= (g_1, \ldots, g_i, s(g_i), g_{i+1}, \ldots, g_k) \text{ if } i \geq 1;\\
  \eta^k_0(g_1, \ldots, g_k) &= (r(g_1), g_1, \ldots, g_k).\end{align*}
 When $k=0$, the map $\eta^0_0: \G\z \to \G^{(1)}$ is just the standard inclusion of $\G\z$ into $\G^{(1)} =\G$.
 \end{defn}
For the definition of a general simplicial space, see e.g. \cite{gen-tu-xu} Section 2.1.

\begin{defn}[cf.~\cite{moerdijk-haefliger, willerton}]
Let $\G$ be a topological groupoid.  A \emph{classifying space} $B\G$ for $\G$ is any space which can be realized as a quotient $B\G = E \G /\G$ of a weakly contractible space $E\G$ by a free action of $\G$.  When we need an explicit model for $B\G$, we will use the geometric realization $|\G_\bullet|$ of the simplicial space associated to $\G$:
\[B\G= |\G_\bullet| = \left( \bigsqcup_{k\geq 0} \G^{(k)} \times \Delta^k \right)/\sim,\]
where $\Delta^k$ denotes the standard $k$-simplex.\footnote{For $k > 0$, $\Delta^k$ can be realized as a subset of $\R^k$, namely, 
\[\Delta^k = \{(t_1, \ldots, t_k): 0 \leq t_1 \leq t_2 \leq \cdots \leq t_k \leq 1\}.\]
If $k =0$, $\Delta^k$ consists of one point, and we will denote $\Delta^0 = \emptyset$.
}
The equivalence relation $\sim$ is defined by $(p, \delta^{k-1}_i v) \sim(\epsilon^k_i p, v)$ for $p \in \G^{(k)}, v \in \Delta^{k-1}$, where $\delta^{k-1}_i: \Delta^{k-1} \to \Delta^k$ is the $i$th degeneracy map, gluing $\Delta^{k-1}$ to the $i$th face of $\Delta^k$, and $\epsilon^k_i: \G^{(k)} \to \G^{(k-1)}$ is the $i$th face map. 
In other words, we have $\delta^0_0(\emptyset) = 0, \delta^0_1(\emptyset) = 1$, and if $k > 1$ 
\[\delta^{k-1}_i(t_1, \ldots, t_{k-1}) = \left\{ \begin{array}{cl}
(0, t_1, \ldots, t_{k-1}) & \text{ if } i =0 \\
(t_1, \ldots, t_i, t_i, t_{i+1}, \ldots, t_k) & \text{ if } 1 \leq i \leq k-1 \\
(t_1, \ldots, t_{k-1}, 1) & \text{ if } i = k.
\end{array}\right.\]

The topology on this model of $B\G$ is the inductive limit 
 topology induced by the natural topologies on $\G^{(n)}, \Delta^n$.  
\end{defn} 


\begin{defn}[\cite{gen-tu-xu} Definition 2.2]
\label{def:principal-bdl-simplicial}
Let $X_\bullet$ be a simplicial space and let $G$ be a topological group.  A \emph{principal $G$-bundle over $X_\bullet$} is a simplicial space $P_\bullet$ such that, for each $k \geq 0$, $P_k$ is a principal $G$-bundle over $X_k$, and the face and degeneracy maps in $P_\bullet$ are morphisms of principal bundles. 
\end{defn}

\begin{rmk}
\label{rmk:gen-morph}
Combining \cite{TXLG} Definition 2.1 and Proposition 2.4 of \cite{gen-tu-xu}, we see that principal $G$-bundles over $\G_\bullet$ are equivalent to generalized morphisms $\G\to G$.
 \end{rmk}
 
 \begin{prop}
 \label{prop:PUn-bdl}
 Let $\G$ be an \'etale groupoid.  Suppose that the classifying space $B\G$ is (homotopy equivalent to) a compact CW complex.   If $\W \to \G$ is a twist of order $n$, then $\W$ gives rise to a principal $PU(n)$-bundle $P \to \G\z$.  Moreover, $P$ admits a left action of $\G$ which commutes with the right action of $PU(n)$.
 \end{prop}
 \begin{proof}
  For any \'etale groupoid $\G$, and any twist $\W \to \G$, Proposition 11.3, Corollary 7.3, and Theorem 8.3 of \cite{brauer-gp-gpoids} combine to tell us that $\W$  determines an element of $H^2(\G, \mathcal{S}^1)$, where $\mathcal{S}^1$ denotes the sheaf of circle-valued functions on $\G\z$. The main Theorem of \cite{moerdijk-haefliger} tells us that we then obtain an associated element $[\W ]$ of $H^2(B\G, \mathcal{S}^1) \cong H^3(B\G, \Z)$.  All of the maps $\text{Tw}(\G) \to H^2(\G, \mathcal{S}^1) \cong H^3(B\G, \Z)$ are group homomorphisms, so  if $\W$ is a torsion twist of order $n$, then $n \cdot [\W]= 0$ also in $H^3(B\G, \Z)$. 

Now, suppose that $B\G$ is a compact CW complex and that $\W$ is a torsion twist of order $n$.  The Serre-Grothendieck theorem (cf. \cite{grothendieck} Theorem 1.6, \cite{donovan-karoubi} Theorem 8 or \cite{lupercio-uribe} Theorem 7.2.11)
 tells us that $\W$  gives rise to a principal $PU(n)$ bundle $Q$ over $B\G$. 

Note that, for  each $k \in \N$, the map $\varphi_k: \G^{(k)} \to B\G$ given by $(g_1, \ldots, g_k) \mapsto [(g_1, \ldots, g_k), (0, \ldots, 0)]$ is continuous.  Moreover, the equivalence relation which defines $B\G$ ensures that the maps $\varphi_k$ commute with the face and degeneracy maps $\epsilon^k_i, \eta^k_i$: 
\[ \forall \ i, \ \varphi_k \circ \eta^{k-1}_i = \varphi_{k-1} \text{ and } \varphi_{k-1} \circ \epsilon^k_i = \varphi_k.\]
Principal $PU(n)$-bundles over a space $X$ are classified by homotopy classes of maps $X \to BPU(n)$, so the maps $\varphi_k$ allow us to pull back our principal $PU(n)$-bundle $Q$ over $B\G$ to a principal $PU(n)$-bundle $P_k$ over $\G^{(k)}$ for each $k \geq 0$. 
Since the maps $\varphi_k$ commute with the face and degeneracy maps for $\G_\bullet$, the maps $\eta^k_i, \epsilon^k_i$ induce morphisms of principal bundles which make $P_\bullet$ into a principal $PU(n)$-bundle over $\G_\bullet$ in the sense of Definition \ref{def:principal-bdl-simplicial}.  Thus, by Proposition 2.4 of \cite{gen-tu-xu}, we have a principal $PU(n)$-bundle $P$ over $\G\z$ which admits an action of $\G$.
 \end{proof}
 
 In what follows, we will combine the bundle $P_\bullet$ constructed above with the canonical $\T$-central extension 
 \begin{equation}
 \beta : \quad 1 \to \T \to U(n) \to PU(n) \xrightarrow{\pi} 1\label{beta}
 \end{equation}
 of $PU(n)$.  
The Leray spectral sequence for the map $BU(n) \to PU(n)$  implies that $\beta$ is a generator of $H^2(PU(n), \T) \cong \Z_n$.  When $n$ is prime, an alternate proof of this fact is given in Theorem 3.6 of \cite{vistoli}.

 These preliminaries completed, we now present the main result of this section.
 \begin{thm}
 \label{main-thm}
 Let $\G$ be an \'etale groupoid.  Suppose that  the classifying space $B\G$ is (homotopy equivalent to) a compact  CW complex. 
 Let $\W \to \G$ be a twist of order $n$ over $\G$ such that the associated $PU(n)$-bundle $P$ of Proposition \ref{prop:PUn-bdl} lifts to a $U(n)$-bundle $\tilde{P}$ over $\G\z$. 
  Then there is a twist $\mathcal{T}$ such that $[\mathcal{T}] = [\W] \in H^2(\G, \mathcal{S}^1)$ and such that $\mathcal{T}$ admits a twisted vector bundle.
 \end{thm}
 \begin{proof}
%
Recall from \cite{moerdijk-haefliger} that for all $s \in \N$,  the inclusion $i: \G_\bullet \to B\G$ induces an isomorphism $i^*_s: H^s(B\G, \T) \to H^s(\G, \T)$, for all $s \in \N$. Moreover, since $i$ is continuous, it also induces a pullback homomorphism $p_1: H^1(B\G, PU(n)) \to H^1(\G_\bullet, PU(n))$, which need not be an isomorphism since $PU(n)$ is not abelian.

Write $v: H^1(B\G, PU(n)) \to H^2(B\G, \T)$ for the Serre map which associates to a principal $PU(n)$-bundle over $B\G$ its Dixmier-Douady class in $H^2(B\G, \T) \cong H^3(B\G,\Z)$. The Serre-Grothendieck Theorem (cf.~\cite{donovan-karoubi} Theorem 8, \cite{lupercio-uribe} Theorem 7.2.11, \cite{grothendieck} Theorem 1.6) establishes that 
\[v: H^1(B\G, PU(n)) \to H^3(B\G, \Z)\] is an isomorphism onto the $n$-torsion subgroup of $H^3(B\G, \Z)$ which is induced by the short exact sequence $\beta$ of Equation \eqref{beta}.

If $P$ is the principal $PU(n)$-bundle over $\G$ which is associated to $\W$ by Proposition \ref{prop:PUn-bdl}, examining the constructions employed in the proof of Proposition \ref{prop:PUn-bdl} reveals that 
\[P = p_1 \circ v\inv \circ (i^*_2)\inv (\W).\]
 
Recall from page 860 of \cite{TXLG} that we have a natural map 
\[ \Phi: H^1(\G_\bullet, PU(n)) \times H^2(PU(n), \mathcal{S}^1) \to H^2(\G, \mathcal{S}^1), \]
which arises from pulling back a principal $PU(n)$-bundle over $\G$ along a $\T$-central extension of $PU(n)$.
We claim that 
\begin{equation}
\Phi(P_\bullet, \beta) = [\W].\label{eq:phi-p-beta}
\end{equation}
Since $\Phi$ is natural, and taking pullbacks preserves cohomology classes, \eqref{eq:phi-p-beta} holds because $v$ is induced by $\beta$, and  $\beta$ generates
$H^2(PU(n), \mathcal{S}^1)$.

We will now use the hypothesis that $P$ admits a lift to a principal $U(n)$-bundle $\tilde{P} \to \G\z$ to show that   $\Phi(P_\bullet, \beta)$ is represented by a twist $\mathcal{T}$ which admits a twisted vector bundle. As explained in \cite{TXLG} pp. 860-1, this hypothesis 
allows us to construct an explicit representative $\mathcal{T}$ of $\Phi(P_\bullet, \beta)$  as follows.

By hypothesis, the quotient map $\pi: U(n) \to PU(n)$ induces a bundle morphism $\tilde{\pi}: \tilde{P} \to P$.  Write $\frac{P \times P}{PU(n)}$ for the gauge groupoid of the bundle $P$, and notice that, if $\rho: P \to \G\z$ is the projection map of the principal bundle $P$, we can define a morphism $\varphi: \G \to \frac{P \times P}{PU(n)}$ as follows.  Given $g \in \G$, choose $p \in P$ with $\rho(p) = s(g)$, and define 
\[\varphi(g) = [g \cdot p, p] .\]
The fact that $P$ is a principal $PU(n)$-bundle implies that $\varphi(g)$ is a well defined groupoid homomorphism.

We define the twist $\mathcal{T} $ over $\G$ by 
\[\mathcal{T} = \{([q_1, q_2], g) \in \frac{\tilde{P} \times \tilde{P}}{U(n)} \times \G : [\tilde{\pi}(q_1), \tilde{\pi}(q_2)] = \varphi(g)\}.\] 
We observe that 
\[([q_1, q_2], g) \in \mathcal{T} \Leftrightarrow g \cdot \tilde{\pi}(q_2) = \tilde{\pi}(q_1).\]
The backward implication is evident; for the forward implication, note that 
\begin{align*} ([q_1, q_2], g) \in \mathcal{T} & \Rightarrow \tilde{\pi}(q_2) \in P_{s(g)}  \Rightarrow \varphi(g) = [g \cdot \tilde{\pi}(q_2), \tilde{\pi}(q_2)].
\end{align*}
But also, $([q_1, q_2], g) \in \mathcal{T} \Rightarrow \varphi(g) = [\tilde{\pi}(q_1), \tilde{\pi}(q_2)].$  Note that 
\[[\tilde{\pi}(q_1), \tilde{\pi}(q_2)] = [p^1, p^2] \Leftrightarrow \ \exists \ u \in PU(n) \text{ s.t. } \tilde{\pi}(q_i) = p^i \cdot u \ \forall \ i;\]
consequently, $g \cdot \tilde{\pi}(q_2) = \tilde{\pi}(q_1)$ as claimed.

The groupoid structure on $\mathcal{T}$ is given by 
\[s([q_1, q_2], g) = s(g), \quad r([q_1, q_2], g) = r(g);\]
if $s(g) = r(h)$ then we define the multiplication by 
\[([q_1, q_2], g)([p_1, p_2], h) = ([q_1 \cdot u, p_2], gh),\]
 where $u \in U(n)$ is the unique element such that $q_2 \cdot u = p_1 \in \tilde{P}$.  

Proposition 2.36 of \cite{TXLG} establishes that $\mathcal{T}$ is a twist over $\G$ such that $[\mathcal{T}] = \Phi(P_\bullet, \beta)$. 
The action of $\T$ on $\mathcal{T}$ is given by 
\begin{equation}
\label{eq:S1-action-T}
z \cdot ([q_1, q_2], g) = ([q_1 \cdot z, q_2], g).\end{equation}

By construction, $\mathcal{T}$ admits a generalized homomorphism $ \mathcal{T} \to U(n)$ which is $\T$-equivariant.  To be precise, the bundle $\tilde{P}$ admits a left action of $\mathcal{T}$: if $\tilde{p} \in \tilde{P}$ lies in the fiber over $s(g)$, and $([q_1, q_2], g) \in \mathcal{T}$, there exists a unique $u \in U(n)$ such that $q_2 \cdot u = \tilde{p}$.  Thus, we define
\[([q_1, q_2], g) \cdot \tilde{p} = q_1 \cdot u .\]
One checks immediately that this action is continuous, $\T$-equivariant, and commutes with the right action of $U(n)$ on $\tilde{P}$.  In other words, the bundle $\tilde{P}$ equipped with this action constitutes a $\T$-equivariant generalized morphism $\mathcal{T} \to U(n)$. 
  Thus, Proposition 5.5 of \cite{TXLG} explains how to construct a $(\G, \mathcal{T})$-twisted vector bundle. 
  Since $[\mathcal{T}] = \Phi(P_\bullet, \beta) = [\W]$, this completes the proof.
\end{proof}
\subsection{Examples}
\label{sec:examples}
In this section, we present some examples establishing that the hypotheses of Theorem \ref{main-thm} are satisfied in many cases of interest.

\begin{example}
\label{ex:BG-cpt}
Let $M$ be a compact CW complex, and let $\alpha$ be a homeomorphism of $M$.  If we set $\G = M \rtimes_\alpha \Z$, the first paragraph of \cite{willerton} Section 1.4.3 tells us that $B\G = M \times_\Z \R$.  Since $M$ is compact, so is $B\G$.
\end{example}

\begin{example} (cf.~\cite{salem-appendix} p.~273)
Let $\mathcal{F}$ be a foliation of a manifold $M$.
The holonomy groupoid $\mathcal{H}_\mathcal{F}$ of  $\mathcal{F}$ is an \'etale groupoid; moreover, if the leaves of the foliation all have contractible holonomy coverings, $B\mathcal{H}_{\mathcal{F}} = M$.  Examples of such foliations include the Reeb foliation of $S^3$ and the Kronecker foliation of $\T^n$. 

In particular, if $M$ is compact, any foliation $\mathcal{F}$ of $M$ with contractible leaves has an associated holonomy groupoid $\mathcal{H}_\mathcal{F}$ with $B\mathcal{H}_{\mathcal{F}}$ compact.
\end{example}

\begin{example}
Let $M := \R P^2 \times S^4$.  We will identify $\R P^2$ with $D^1/\sim$, where (in polar coordinates) $D^1 = \{(\rho, \theta) \in \R^2: 0 \leq \theta < 2\pi, 0 \leq \rho \leq 1\}$ and $(1, \theta) \sim (1, \theta + \pi)$.

Fix $x \in \R \backslash \Q$, and consider the homeomorphism $\alpha$ of $\R P^2 \times S^4$ given by 
\[\alpha([\rho, \theta], z) = ([\rho, \theta + (1-\rho) x], z).\]  Let $\G = M \rtimes_\alpha \Z$.  Since $M$ is compact, Example \ref{ex:BG-cpt} tells us that $B\G$ is compact as well.

By the K\"unneth Theorem, $\Z/2\Z \cong H^2(\R P^2, \Z) \otimes H^0(S^4, \Z)$ is a subgroup of $H^2 (M, \Z) \cong H^1(M, \T)$.  The groupoid $\G$ is an example of a Renault-Deaconu groupoid (cf.~\cite{deaconu-RD, deaconu-kumjian-muhly, ionescu-muhly}); thus, by Theorem 2.2 of \cite{deaconu-kumjian-muhly}, twists over $\G = M \times_\alpha \Z$ are classified by $H^1(M, \T)$.  It follows that $\G$ admits nontrivial torsion twists.

The short exact sequence $1 \to \T \to U(n) \to PU(n) \to 1$ tells us that the obstruction to a principal $PU(n)$-bundle over $M$ (an element of $H^1(M, PU(n))$) lifting to a principal $U(n)$-bundle over $M$ lies in $H^2(M, \T) \cong H^3(M, \Z) $.  However, the K\"unneth Theorem tells us that 
\[H^3(M, \Z) \cong H^3(\R P^2, \Z) \otimes H^0(S^4) \cong H^3(\R P^2, \Z) = 0.\]
 In other words, every principal $PU(n)$-bundle over $M$ lifts to a principal $U(n)$-bundle over $M$, so every torsion twist over $\G = M \times_\alpha \Z$ satisfies the hypotheses of Theorem \ref{main-thm}.
 
 Furthermore, since the action of $\Z$ on $M$ is not proper, this Example lies outside the cases (cf.~\cite{dwyer, TXLG}) where it was previously known that torsion twists admit twisted vector bundles.
\end{example}
\bibliographystyle{amsplain}
\bibliography{eagbib}
\end{document}